\documentclass[12pt,a4paper]{amsart}
\usepackage[utf8]{inputenc}
\usepackage[T1]{fontenc}
\usepackage{amsmath,amsfonts,amssymb,amsthm,mathrsfs}
\usepackage{color}
\usepackage{marginnote}

\DeclareMathOperator{\sgn}{sgn}

\newcommand{\PM}{\mathcal{PM}}
\newcommand{\XPM}{\mathcal{X}}

\newcommand{\Projector}{\mathbb{P}}
\newcommand{\RRR}{\mathbb{R}^{3}}
\newcommand{\vf}{{\varphi}}
\newcommand{\bbfR}{\mathbb{R}}

\newcommand{\be}{\begin{equation}} 
\newcommand{\ee}{\end{equation}}
\newcommand{\bea}{\begin{eqnarray}} 
\newcommand{\eea}{\end{eqnarray}}
\newcommand{\bean}{\begin{eqnarray*}} 
\newcommand{\eean}{\end{eqnarray*}}
\newcommand{\rf}[1]{(\ref {#1})}

\newcommand{\ud }{\,{\rm d}}

\DeclareMathOperator*{\esssup}{ess\,sup}

\def\dx{\,{\rm d}x}

\def\ve{\varepsilon}

\def\mn{|\!\!|}

\def\mn2{|\!\!|_{M^{d/2}}}

\def\X{{\mathcal X}}

\newtheorem{theorem}{Theorem}
\newtheorem{proposition}[theorem]{Proposition}
\newtheorem{lemma}[theorem]{Lemma}
\newtheorem{corollary}[theorem]{Corollary}

\theoremstyle{definition}

\theoremstyle{remark}
\newtheorem{remark}[theorem]{Remark}

\numberwithin{equation}{section}
\numberwithin{theorem}{section}

\author[M. Cannone]{Marco Cannone}
\address[M. Cannone]{
Laboratoire d'Analyse et de Mathématiques Appliquées,
Universit\'e Gustave Eiffel,
5 boulevard Descartes,
Bâtiment Copernic,
77420 Champs-sur-Marne, France.}
\email{marco.cannone@univ-eiffel.fr}

\author[G. Karch]{Grzegorz Karch}
\address[G. Karch]{
 Instytut Matematyczny, Uniwersytet Wroc\l awski,
 pl. Grunwaldzki 2/4, \hbox{50-384} Wroc\-\l aw, Poland}
\email{Grzegorz.Karch@math.uni.wroc.pl}

\author[D. Pilarczyk]{Dominika Pilarczyk}
\address[D. Pilarczyk]{ Wydzia{\l} Matematyki, Politechnika Wroc{\l}awska,
Wybrze\. ze Wyspia\'nskiego 37, Wroc{\l}aw, Poland} 
\email{dominika.pilarczyk@pwr.edu.pl}

\author[G. Wu]{Gang Wu}
\address[G. Wu]{
School of Mathematical Sciences, University of Chinese Academy of Sciences, Beijing 100049, 
People's Republic of China}
\email{wugang2011@ucas.ac.cn}

\thanks{The research of Gang Wu was partially supported by the National Natural Science
Foundation of China under grant No. 11771423.}

\title[Navier-Stokes equations with singular external force]{
Stability   of singular solutions \\
to the Navier-Stokes system}

\begin{document}
\begin{abstract}
We develop mathematical
methods 
which allow us to study asymptotic properties of  solutions 
to the three dimensional Navier-Stokes system for incompressible fluid in the whole three dimensional space.
We deal either with the Cauchy problem or with the stationary problem where solutions may be singular due to singular external forces which 
 are either singular finite measures or more general tempered distributions with bounded Fourier transforms. 
We present results on asymptotic properties of such solutions either for large values of the space variables (so called the far-field asymptotics) or for large values of time.
\end{abstract}



\date{\today}
\maketitle

\baselineskip=16pt

\section{Introduction}

\subsection*{Singular solutions to the Navier-Stokes system}

While a  proof of a regularity of  weak energy  solutions to an initial-boundary value problem for 
 the Navier-Stokes system 
 seems to be still out of reach, several advances were made  in the opposite direction of the study of singular (and sometimes exact) solutions for these equations.
In fact, 
as observed by Heywood \cite{H94}, ``\textit{it is easy to construct a singular solution of the Navier-Stokes equations that is driven by a singular force. One simply constructs a solenoidal vector field u that begins smoothly and evolves to develop a singularity, and
	then defines the force to be the residual.}'' Such singular solutions satisfy, usually in a distributional sense, the incompressible Navier-Stokes system 
\begin{equation} \label{1}
u_{t}+(u\cdot \nabla )u-\Delta u+\nabla p=f, \qquad \nabla\cdot u=0, 
\qquad x\in\RRR, \, t>0,
\end{equation}
where an external force $f=f(\cdot,t)$ is singular for some values of time $t$.
The purpose of this paper is to provide a systematic approach to the study of a large space and time behavior for such singular solutions. 
More precisely, we develop mathematical tools which allow us to study asymptotic properties of   solutions corresponding to external  forces which  are either singular finite measures or more general tempered distributions with bounded Fourier transforms. 
The following three examples
  of singular solutions fall within the scope of this work.

The Sl\"ezkin-Landau solutions  
$\left\{U^\beta(x),P^\beta(x)\right\}_{\beta \in \RRR}$  satisfy  the stationary Navier-Stokes system with the singular external force 
\begin{equation}\label{eqLandau}
    \begin{aligned}
        -\Delta U^{\beta}+\left(U^{\beta}\cdot \nabla \right)U^{\beta}+\nabla P^\beta &=\beta\delta_{0} , \qquad x\in \RRR \\
        \nabla\cdot U^{\beta}&=0,
    \end{aligned}
\end{equation}
where $\delta _0$ stands for the Dirac measure supported at the origin and $\beta\in\RRR$ is a constant vector. Such solutions  satisfy the system  in a weak sense and
in a particular case of $\beta = (\beta_1,0,0)$, they  are 
given by the following explicit formulas
\begin{equation}\label{L sol}
\begin{aligned}
   U^{\beta}_1(x)&=2\frac{c|x|^2-2x_1|x|+cx_1^2}{|x|(c|x|-x_1)^2},  &U^{\beta}_2(x)&=2\frac{x_2(cx_1-|x|)}{|x|(c|x|-x_1)^2},\\
   U^{\beta}_3(x)&=2\frac{x_3(cx_1-|x|)}{|x|(c|x|-x_1)^2},   &P^{\beta}(x)&=4\frac{cx_1-|x|}{|x|(c|x|-x_1)^2},
\end{aligned}
\end{equation}
where $|x|=\sqrt{x_1^2+x_2^2+x_3^2}$.
 The parameter  $c$ in expressions \eqref{L sol}
 is an arbitrary constant such that $|c|>1$ and is related to the parameter 
 $\beta_1 \in \bbfR$  by the formula 
\begin{equation}\label{bc}
 \beta_1=  \beta_1 (c)=\frac{8\pi c}{3(c^2-1)}\Big( 2+6c^2 -3c(c^2-1)\log\Big(\frac{c+1}{c-1}\Big) \Big).
\end{equation}
Notice that  the function $\beta_1=\beta_1 (c)$ is decreasing on $(-\infty , -1)$ and $(1, +\infty )$. Moreover, $\lim_{c\rightarrow  1}\beta_1 (c)=+\infty $, $\lim_{c\rightarrow  -1}\beta_1 (c)=-\infty $ and $\lim_{|c|\rightarrow  \infty }\beta_1 (c)=0$.

These explicit stationary solutions seem to be first calculated by 
Sl\"ezkin \cite{S34} (see Appendix in \cite{G09} for a translation of this paper),
rediscovered by 
Landau \cite{L44}, and reported and discussed in other   textbooks and papers,
see e.g.~\cite[p. 82]{LL59}, \cite[p. 206]{B99},  \cite{S51,CK04,TX98}.
Here, we would like to note that the construction of the Sl\"ezkin-Landau solutions  from the work \cite{TX98} inspired the authors of \cite{CK04} to write their paper.
Recently, \v{S}ver\'ak \cite{S11} proved  that the Sl\"ezkin-Landau solutions \rf{L sol} are the only stationary solutions which are invariant under the natural scaling of the equations, namely under the scaling 
\begin{equation}\label{scaling}
u(x) \rightarrow \lambda u(\lambda x)\quad \text{and} \quad p(x)\rightarrow \lambda^2 p(\lambda x)\quad  \text{for} \quad \lambda >0.
\end{equation}

A time dependent version $U^\gamma=U^\gamma(x,t)$  of the Sl\"ezkin-Landau solutions \eqref{L sol}
has been constructed in the paper \cite{KZ15} by considering 
  the Cauchy problem 
\begin{equation}\label{NSsing}
	\begin{aligned}
		U^\gamma_t - \Delta U^\gamma +(U^\gamma\cdot  \nabla)U^\gamma+\nabla P^\gamma &= \beta \delta_{\gamma},
		\qquad x\in\RRR, \, t>0,\\
		\nabla \cdot U^\gamma &=0,\\
 U^\gamma(0)&=0
\end{aligned}
\end{equation}
with constant $\beta\in \bbfR^3$, the shifted Dirac delta measure $\delta_{\gamma(t)} =\delta(\cdot -\gamma(t))$, and a H\"older continuous function $\gamma :[0,\infty)\to \RRR$. It was shown in \cite{KZ15} that
the  solutions  $U^\gamma=U^\gamma(x,t)$ and $P^\gamma=P^\gamma(x,t)$ to problem \rf{NSsing} are locally bounded away from the graph of the curve 
$\Gamma=\{(\gamma (t),t)\in \RRR\times [0,\infty)\,:\, t\geqslant 0\}$  and are singular along this curve, see Remark \ref{rem:q>3} below for more comments on these solutions. 
Moreover, it was proved in  the recent paper \cite{KSS20} that singular solutions of problem \rf{NSsing} have finite energy, namely,  
$U^\gamma \in C\big([0,\infty),L^2(\RRR)^3\big)$.

Another class of singular stationary solutions has been considered recently in a series of papers \cite{LLY_18_II,LLY_18_I,LLY_19}. Those solutions are invariant under the scaling \rf{scaling}, they are axisymmetric, and they belong to the space $C^\infty(\RRR\setminus \{(x_1, x_2) = 0\})$
with a possible singular ray $\{(x_1, x_2) = 0\}$. It was shown in the recent preprint \cite{LY19} that
such singular solutions  satisfy, for a certain range of parameters,  
the stationary Navier-Stokes system 
\begin{equation}\label{egLLY}
    \begin{aligned}
	-\Delta U+\left(U\cdot \nabla \right)U+\nabla P &=F , \qquad x\in \RRR ,\\
	\nabla\cdot U&=0,
\end{aligned}
\end{equation}
where $F$ is a tempered distribution defined by the formula
\begin{equation}\label{sin_force}
F\vf=	\left( 4\pi c\int_{-\infty}^{\infty} \log |x_3| \partial_{x_3}\vf (0,0,x_3)\ud x_3 - b\vf (0)\right)e_3
\end{equation}
for every $\vf \in \mathcal{S}(\RRR)$. Here, $c, b\in \bbfR$ are suitable constants and $e_3=(0,0,1)$. 
It is well-known that the Fourier transform of the Dirac measure $\delta_0 \varphi =\varphi(0)$ is give by the constant $(2\pi)^{-3/2}$
(with the Fourier transform defined as in \eqref{FTr}). It follows from  
Lemma~\ref{Lem:Tf} below 
 that the Fourier transform of the distribution $F$ defined in \eqref{sin_force} is represented by the bounded function
$$
\widehat F(\xi)=\left(4\pi c 2^{-\frac{3}{2}} \pi^{-\frac{1}{2}}i \sgn \xi_3
-b (2\pi)^{-\frac32}
\right) e_3.  
$$


\subsection*{Cauchy problem in scaling invariant spaces}
Before we present the results from this work,  
we recall related  results on the existence of global-in-time solutions to the Cauchy problem \eqref{1} for the Navier-Stokes system in $\bbfR^3$ in  scaling invariant spaces.
Here, a~Banach space $X$ is  called scaling invariant  if its norm satisfies
$\|v\|_{X}=\|\lambda v(\lambda \cdot)\|_{X}$ for all $v\in X$ and all $\lambda>0$.  
There is a rich literature about existence of global-in-time solutions with small initial data in scaling invariant  spaces such as homogeneous Sobolev space $\dot H^{1/2}(\bbfR^3)$, the Lebesgues space $L^3(\bbfR^3)$, the Marcinkiewicz space $L^{3,\infty}(\bbfR^3)$, and 
the Besov, Morrey,   weak-Morrey, Fourier–Besov, weak-Herz, Fourier–Herz, Besov–Morrey, 
$BMO^{-1}$, ...  spaces.
 We refer the reader to the review \cite{C04} and to the monograph \cite{LR16} 
 for a discussion  of those results and for  references.
 Most of those existence results are proved by using the Kato approach \cite{K84} that consists in applying  the Banach  fixed point argument in a suitable time-dependent space which norm is composed of two parts, where  the usual norm of the persistence space $L^\infty\big((0, \infty); X\big)$ 
 is supplemented  the  auxiliary time-weighted norm  $\sup_{t>0} t^\alpha  \|u(\cdot,t)\|_Y$ for   another Banach space  $Y$.
 
 \subsection*{Navier-Stokes system with external forces}
 
 In order to apply the Kato approach to construct small global-in-time solutions to the Cauchy problem for system \eqref{1},
 one should consider a force $f=f(x,t)$ from another Banach  space with a norm  invariant under the scaling  
 \begin{equation}\label{f:scal}
  f_\lambda (x, t) =\lambda^3 f(\lambda x, \lambda^2t).
  \end{equation} 
 Such results are contained, for example,  in the paper  \cite{CP99}, where $ f={\rm div} F$ 
 with $F\in L^s\big((0, \infty), L^q(\bbfR^3)\big)$
 for $2/s +3/q=2$ with $2/3 <q<\infty$. 
 More recently,  global-in-time solutions have been constructed to this problem in the case of external forces from the  Lorentz space $L^{s,\infty}\big((0, \infty), L^{q,\infty}(\bbfR^3)\big)$
 for $2/s +3/q=3$ with $1 <q<\infty$ (\cite{KS18}), 
 from time-weighted Besov spaces (\cite{KS18a}),
 and from the homogeneous Besov space with both negative and positive differential orders
 (\cite{KS20}).
 
 The Kato two-norm approach cannot be applied  to solutions which do not decay as $t\to\infty$, for example, in  the case of the Navier-Stokes system \eqref{1} with a time independent external force. Here, global-in-time solutions corresponding to small initial conditions and small external forces  should be constructed in a scaling invariant space 
 $L^\infty \big((0,\infty), X\big)$,
 where only one natural norm $\sup_{t>0} \|\cdot\|_X$
 is involved. 
  To the best of our knowledge, 
 the Marcinkiewicz space $L^{3,\infty}(\bbfR^3)$  (\cite{M97,Y00}), 
the space $ {\PM}^2$    (\cite{LJS97,CK04}),
  the Fourier-Besov space $F\dot B^{2-\frac3p}_{p,\infty}$ with $p >3$ (\cite{KY11}), 
the  weak-Herz spaces (\cite{T11}),
and the weak Morrey space  ${\mathcal M}_{p,\infty,3-p}$ with $p\in (2,3]$ (\cite{LR15})
 are the only spaces where such an one norm approach  is possible.
 In general, solutions obtained in this setting  are only time-weakly continuous at $ t =0$
 because of the lack of strong continuity at $t =0$ of the heat semigroup 
 $\{S(t)\}_{t\geq 0}=\{e^{t\Delta}\}_{t\geq 0}$ in $X$. 
 
In this work, we choose an approach from \cite{CK04} where small solutions of the Cauchy problem for system \eqref{1} have been constructed in the space $ {\PM}^2$. We are convinced, however, that our ideas presented in this work can be extended to solutions from above mentioned spaces, where one norm approach to construct global-in-time  solutions is possible.

  \subsection*{Stationary solutions}
  In this work, we deal also with stationary solutions satisfying the system
 $$ 
         (w\cdot \nabla )w-\Delta w+\nabla p = g ,\quad 
         \nabla\cdot w  =0, \quad 
         x\in \RRR,
$$ 
where the existence and stability of  solutions have been studied in several works either in  the
case of bounded domains or exterior domains, see {\it e.g.}~the monograph \cite{G11}. 
Here, we contribute to the theory on 
 the existence and the stability of solutions to this stationary system
considered in the whole space $\RRR$ and 
with a singular and sufficiently small (in a suitable sense) external force, see {\it e.g.}~\cite{S88, KY95, BBIS11,BS09,DI17,PP13} and references therein.

\subsection*{Summary of results from this work}

As starting point in this work, we consider 
a~solution $u\in  C_w\big((0, \infty); \PM^2\big)$
to the Cauchy problem for the three dimensional Navier-Stokes system \eqref{1} 
 with a possibly singular external force $f\in C_w\big((0, \infty); {\PM}^0\big) $, see below for definitions of these spaces. 
 Such  solutions are constructed 
in Propositions~\ref{thmNSExistence} 
in the case of the Cauchy problem
and in Proposition \ref{thmENSExistence} in the case of stationary solutions  
which we recall from  \cite{CK04}.
The use of these spaces allows us to simplify several calculations of this work and we can deal with external forces which are tempered distributions with bounded Fourier transforms. 
In Theorem \ref{corNSLandau}, we show that  every solution
from $C_w\big((0, \infty); {\PM}^2\big)$
 behaves, for large values of $|x|$ (which is measured by the Lebesgue norm $L^q(\bbfR^3)$ with $q\in [2,3)$)
 as a solution to the heat equation. 
In Theorem~\ref{lemENSRegular}, we consider stationary solutions and we show their stability in the norms of $L^q(\bbfR^3)$ with some $q<3$.
A general asymptotic stability result of solutions to the Cauchy problem for system 
\eqref{1} is presented in Theorem \ref{thm_self_similar}.
In  remarks,  we explain  how our theorems are related to previous results and, especially, we emphasize how they can be applied to singular solutions discussed above.


\section{Results and comments}\label{results}

We consider  mild solutions to  the Cauchy problem  
\begin{equation}\label{eqNSNS}
\begin{aligned}
        u_{t}+(u\cdot \nabla )u-\Delta u+\nabla p&=f,\qquad (x,t)\in \RRR\times\mathbb{R}^{+},\\
        \nabla\cdot u&=0,\\
        u(x,0)&=u_{0}(x).
    \end{aligned}
\end{equation}
Namely, applying the Duhamel principle, we deal with the corresponding integral equation
\begin{equation}\label{int_eq}
u(t) = S(t)u_0-B(u,u)(t)+F(t)
\end{equation}
with the heat semigroup 
\begin{equation}\label{heat_semigroup}
S(t) u_0(x)=
\int_{\bbfR^3}
\frac{1}{(4\pi t)^{3/2}}\textrm{exp}\left(-\frac{|x-y|^2}{4t}\right) u_0(y)\ud y,
\end{equation}
 the bilinear form
\begin{equation}\label{bili_form}
  B(u,v)(t) \equiv \int_{0}^{t}\Projector\nabla S(t-\tau)\cdot\big(u(\tau)\otimes v(\tau)\big)\ud\tau,
\end{equation}
and the term obtained from the external force
\begin{equation}\label{force}
 F(t)\equiv \int_{0}^{t}S(t-\tau)\Projector f(\tau)\ud\tau.
\end{equation}
The Leray projector $\Projector$ in formulas \eqref{bili_form} and \eqref{force}
is defined as the Fourier multiplier 
with the matrix component symbol $(\widehat{\Projector}(\xi))_{j,k}=\delta_{jk}-{\xi_{j}\xi_{k}}/{|\xi|^{2}}$ satisfying the obvious inequalities $|(\widehat{\Projector}(\xi))_{j,k}| \leqslant 1$ for all $\xi\in \bbfR^3\setminus \{0\}$ and all $j,k \in \{1,2,3\}$. Moreover, for two vector fields $u=(u_1,u_2,u_3)$ and $v=(v_1,v_2,v_3)$, we introduce  their tensor product 
$u\otimes v$ which is  the $3\times 3$ matrix whose $(\ell,k)$ entry is $u_\ell v_k$. Here and in what follows the Fourier transform of an integrable function $f$ is given by 
\begin{equation}
\label{FTr}
	\widehat{f}(\xi) = (2\pi)^{-\frac{3}{2}}\int_{\RRR}e^{-i x \cdot \xi}f(x) \ud x.
\end{equation}

The following functional spaces
have been introduced in the work \cite{CK04} and
are systematically used in this paper
\begin{equation}\label{PM_space}
	\PM^{a}\equiv \{u\in \mathcal{S}^{\prime}(\RRR): \widehat{u}\in L^{1}_{loc}(\RRR), \|u\|_{\PM^{a}}=\esssup_{\xi\in \RRR}|\xi|^{a}|\widehat{u}(\xi)|<\infty\},
\end{equation}
where $a\in \bbfR$ is a given parameter.
We also consider the counterparts of these spaces for time dependent tempered distributions 
\begin{equation*}
	\mathcal{X}^{a}\equiv C_{w}\big([0,\infty),\PM^{a}\big)
	\quad \text{with the norm} \quad \| u\|_{\mathcal{X}^{a}} =\esssup_{t\in [0,\infty)}\|u(\cdot )\|_{\PM^a}.
\end{equation*}
Here, the symbol $C_w$ denotes distributions $u(\cdot, t)$ weakly continuous in time
which means that the mappings $t\mapsto \langle u(t),\vf \rangle_{\mathcal{S}'(\RRR)}$ are continuous for each $\vf \in \mathcal{S}(\RRR)$ (the usual Schwartz class of test functions).

First, we recall a result on the existence of small solutions to the integral equation \eqref{int_eq} in the space $\PM^{2}$. 

\begin{proposition}[{\cite[Thm.~4.1]{CK04}}]\label{thmNSExistence}
There is $\ve >0$ such that if $\|u_{0}\|_{\PM^{2}}<\ve$ and $\|f\|_{\mathcal{X}^0}<\ve$ then there exists a global-in-time mild solution $u\in \XPM^{2}$ to the Cauchy problem for the Navier-Stokes system \eqref{eqNSNS}.  Moreover, there exists a constant $C>0$ such that this solution satisfies
$\|u\|_{\mathcal{X}^2}\leqslant C\big(\|u_0\|_{\PM^2}+\|f\|_{\mathcal{X}^0}\big)$.
\end{proposition}

This proposition  is obtained from the Banach contracted principle 
applied to the ``quadratic'' equation \eqref{int_eq}. We recall that reasoning in the beginning of Section \ref{proofs}.

\begin{remark}
Let us  comment the functional spaces used in Proposition \ref{thmNSExistence}. We consider external forces $f=f(\cdot,t) $
which are time 
 dependent tempered distributions 
 satisfying $\esssup_{t>0,\xi\in\bbfR^3} |\widehat f(\xi, t)|<\infty$
 in order to obtain  solutions in the space $ \mathcal{C}_{w}\big([0,\infty),\PM^{2}\big)$.
  However, 
due to   the embedding  $\PM^2\subset L^{3,\infty}(\bbfR^3)$
 (see Lemma \ref{KPS-lem} below)
 as well as the 
 the 
 embedding  
 into the homogeneous Besov space
$\PM^2\subset \dot B^{-1+3/p,\infty}_p(\bbfR^3)$ with $p>3$. 
 (see  \cite[Lemma 7.1]{CK04}),
  solutions from Proposition \ref{thmNSExistence}
agree with mild solutions of problem \eqref{eqNSNS} constructed in other scaling invariant spaces and discussed in several works, see the review \cite{C04} 
and the monograph \cite{LR16}
 for other references.
\end{remark}

Our first result of this work states  that a behavior of mild solutions 
to problem \eqref{eqNSNS}
obtained in Proposition~\ref{thmNSExistence}
 is  determined in the scale of $L^p$-spaces 
 by the quantity $S(t)u_0+F(t)$.
Notice that we do not impose any smallness assumption  in Theorem \ref{corNSLandau} below neither on initial data nor external forces.

\begin{theorem}\label{corNSLandau}
Every mild solution $u\in \mathcal{X}^2$ of problem \rf{eqNSNS} corresponding to an initial datum 
$u_0\in \PM^{2}$ and an external force $f \in \mathcal{X}^0$ satisfies
$
 u(t)- S(t)u_0-F(t) \in L^{q}(\RRR)
 $
 for each  $q\in \left[2,3\right)$ and all  $t>0$.
Moreover, for every $q\in \left[2,3\right)$  there exists a constant $C=C(q)>0$ such that 
\begin{equation}\label{uSF:Lq}
\| u(t)-S(t)u_0-F(t)\|_q \leqslant Ct^{\frac{3-q}{2q}}
 \|u\|^2_{\mathcal{X}^2}
 \qquad \text{for all} \quad t>0.
\end{equation}
\end{theorem}

\begin{remark}
In fact, in the proof of  Theorem \ref{corNSLandau}, we show the following more general estimate  valid  for each $b\in \left[0,2\right]$:
\begin{equation}\label{uSF:PM}
\|u(t) - S(t)u_0-F(t)\|_{\PM^b}  \leqslant 
Ct^{\frac{2-b}{2}}\|u\|_{\mathcal{X}^2}^2
\quad \text{for all} \quad t>0.
\end{equation}
Then, the $L^q$-estimate \eqref{uSF:Lq} is obtained from inequality \eqref{uSF:PM} with $b\in \{0,2\}$ combined with the 
 interpolation inequality  from Lemma \ref{lemEmbedding} which requires $q\geq 2$. 
On the other hand,  inequality \eqref{uSF:PM} with $b\in [0,1)$ may be considered as a counterpart of estimate \eqref{uSF:Lq} with $q\in [1,2)$.
\end{remark}

\begin{remark}We prove below in Corollary \ref{cor:F:Lq} that, for each  $f\in \X^0$, we have 
$F(t)\in L^q(\RRR)$ for all $q\in [2,3)$ and all $t>0$.
On the other hand, in general, $S(t)u_0\notin L^q(\RRR)$ for $q\in [2,3)$
if $u_0\in \PM^2$ is homogeneous of degree $-1$,  see {\it e.g.}~\cite{BV07}.
Thus, Theorem \ref{corNSLandau} states that a behavior of solutions of problem \rf{eqNSNS} with a singular initial datum from $\PM^2$ and with a singular external force from $\X^0$ is determined 
 for large values of $|x|$
only by $S(t)u_0$ at the first approximation, see {\it e.g.}~\cite{BM02,B09, BV07} for analogous results in the case of $f=0$. 
The external force $f=f(\cdot,t)$ will appear in a higher order far field asymptotic expansion of solutions which was already observed in the work \cite{BB09,BBJ09} in the case of more regular initial conditions and external forces.
 \end{remark}

\begin{remark}
Recall  that $S(t)u_0\in C\big([0,\infty), L^2(\RRR)^3\big)$ for each $u_0 \in  L^2(\RRR)^3$. Thus, 
Theorem \ref{corNSLandau} implies that
in the case of arbitrary $u_0\in L^2(\RRR)^3\cap \PM^2$ and possibly singular $f\in\X^0$,
 the corresponding solution has a finite energy, namely,   $u(t)\in L^2(\RRR)^3$ for each $t\geqslant 0$. 
Such a finite energy result was recently obtained  in \cite{KSS20} in the particular case of 
singular solutions to problem  \eqref{NSsing}.
\end{remark}

\begin{remark}\label{rem:q>3}
By methods from this work,
it is also possible to study asymptotic properties of solutions 
of problem \rf{eqNSNS} in the $L^q$-spaces with 
 $q>3$. Here, however, we obtain an asymptotic profile of a solution around its singular points as, for example, 
in the work \cite{KZ15} where singularities of solutions to problem \eqref{NSsing} have been studied.
Assuming $\gamma$ to be H\"older continuous with exponent $\alpha\in(\frac{1}{2},1]$, the paper \cite{KZ15}  compares the solution $(U^\gamma,P^\gamma)$ of problem  \eqref{NSsing} to the Sl\"{e}zkin-Landau solutions made time dependent by translating the origin to $\gamma(t)$ proving that for all $t>0$,
$$U^\gamma(t) -U^c(\cdot-\gamma(t)) \in L^q(\RRR)\qquad  \text{for}\quad  3<q<\frac {3}{2(1-\alpha)}
$$
and
$$
P^\gamma(t)-P^c(\cdot-\gamma(t))\in L^q(\RRR) \qquad \text{for}\quad  \frac {3}{2}<q<\frac {3}{3-2\alpha}.
$$
An analogous result describing  
singularities of very weak stationary solutions of the Navier-Stokes system
was obtained in \cite{MT12}.  
We  study this case  in Theorem \ref{lemENSRegular} below devoted to the stationary Navier-Stokes system.
\end{remark}

Next, we present 
a counterpart of Theorem \ref{corNSLandau} in the case of 
 a solution $w=w(x)$ to the stationary Navier-Stokes system
\begin{equation}\label{eqENSNS}
         (w\cdot \nabla )w-\Delta w+\nabla p = g ,\quad 
         \nabla\cdot w  =0, \quad 
         x\in \RRR
\end{equation}
with a given external force $g=g(x)$.
In Section \ref{estimates},
we define (via the Fourier transform) the bilinear operator 
\begin{equation}\label{bili_form_E}
  B_{E}( w, v) \equiv \int_{-\infty}^{t}\Projector\nabla S(t-\tau)\cdot\big( w\otimes v\big)\ud  \tau
\end{equation}
and 
\begin{equation}\label{F_E}
G \equiv \int_{-\infty}^t S(t-\tau )\Projector g \ud  \tau
\end{equation}
which allow us to formulate 
system \rf{eqENSNS}  (see e.g. \cite[Thm.~1.1]{Y00}, \cite[Prop.~6.1]{CK04})
 in the formally equivalent form 
\begin{equation}\label{int_E}
 w = -B_E(w,w) +G. 
\end{equation}
As in the case of the Cauchy problem \rf{eqNSNS}, we study properties of solutions to integral equation \rf{int_E} which we also call as mild solutions to system \eqref{eqENSNS}.
The following result recalled from \cite{CK04} gives an existence of mild stationary  solutions in the space $\PM^2$. 

\begin{proposition}[{\cite[Thm.~6.1]{CK04}}]\label{thmENSExistence}
There exists $\ve >0$ such that, for every $g\in \PM^{0}$, if $\|g\|_{\PM^{0}}<\ve $ then system \rf{eqENSNS} has a solution $w\in \PM^{2}$.
Moreover, there exists a constant $C>0$ such that this solution satisfies the inequality 
$\|w\|_{\PM^2}\leqslant C\|g\|_{\PM^0}$. 
\end{proposition}

\begin{remark}
Let us quote other related works where small solutions to the stationary Navier-Stokes system \eqref{eqENSNS} have been constructed for sufficiently small and possibly rough external forces.
In \cite{KY95},  an external force belongs to   
$\mathcal{M}^{-2}_{3,r}$, 
$r\in (2,3]$, which is  the Sobolev-type space based on the Morrey space. 
In \cite{BBIS11},  $\Delta^{-1}\Projector g\in L^{3,\infty}(\RRR)$.
In the work \cite{PP13},  $\Delta^{-1}g \in \mathcal{V}^{1,2}$ --  a suitable capacity space which can be embedded into the homogeneous Morrey space $\mathcal{M}^{2,2}$. 
The authors of   \cite{DI17} deal with  $g = \phi g_0+g_1$, where $g_0$ is homogeneous of degree $-3$, $g_1$ is bounded and decays sufficiently fast at infinity and $\phi$ is a smooth cut-off function such that $\phi(x)=0$ if $|x|\leqslant 1/2$. 
\end{remark}

Our next goal is  show a certain kind of stability of stationary solutions 
from Proposition~\ref{thmENSExistence}
in the class of stationary solutions.

\begin{theorem}\label{lemENSRegular}
Let $\varepsilon >0$ be as in Proposition \ref{thmENSExistence}
and consider solutions $w_{i}\in\PM^2$ with $i\in \{1,2\}$ to stationary system 
\rf{eqENSNS} corresponding to external forces $g_{i}\in \PM^0$ satisfying $\|g_i\|_{\PM^0}<\varepsilon$. 
Assume that $g_1-g_2\in \PM^{b-2}$
for some  $b\in \left(1,3\right)$.
There exists $\tilde\varepsilon \in (0,\varepsilon]$ such that if
 $\|g_i\|_{\PM^0}\leqslant \tilde\varepsilon$
then $w_1-w_2\in \PM^b$.
Moreover, 
\begin{enumerate}
\item
if $b\in \left(1,2\right)$, 
 then   
 $w_1 -w_2\in  L^{q}(\RRR)$
  for each
    $q\in \left(\tfrac{3}{3-b},3\right)$
  such that  $q\geqslant 2$;
  

 \item
  if $b\in (2,3)$, then $w_1 -w_2\in  L^{q}(\RRR)$
  for each
  $q\in \left(3,\tfrac{3}{3-b}\right)$.

\end{enumerate}
In both case, the following inequality holds true
\begin{equation}\label{w1w2:3}
  	\| w_1 - w_2\|_q\leqslant
  	C\| g_1-g_2\|_{\PM^0}^{1-\frac{q-3}{q(b-2)}} \| g_1-g_2\|_{\PM^{b-2}}^{\frac{q-3}{q(b-2)}}
  \end{equation}
   for a constant $C>0$.
\end{theorem}

\begin{remark} Let us illustrate how to use Theorem \ref{lemENSRegular}
in the case of $b\in (1,2)$
  to study a far field asymptotic behavior of stationary solutions. 
Consider 
\begin{equation}\label{g1:force}
g_1\in L^1(\RRR)^3\subset \PM^0
\quad \text{satisfying}
\quad  \int_{\RRR}|x|^{2-b}|g_1(x)|\dx  <\infty
\end{equation}
for some   $b\in (1,2)$ and
$$
 g_2=\beta \delta_0 \quad 
\text{with}\quad  \beta  =\int_{\RRR} g_1 (x)\dx \in \RRR.$$
More generally, we can choose $g_1$ as a finite measure on $\RRR$ with a finite moment of order $2-b$.
For these two external forces, assumption 
$g_1-g_2\in \PM^{b-2}$
is satisfied
by Proposition \ref{est:moment}, below.
Thus, by Theorem  \ref{lemENSRegular}, under a suitable smallness assumption, the corresponding solutions $w_1$ and $w_2$ 
of the stationary system \rf{eqENSNS}
satisfy 
 inequality \eqref{w1w2:3}.
In the particular case of $\beta=(\beta_1,0,0)$, the solution $w_2$ agrees with the explicit Sl\"ezkin-Landau solution \eqref{L sol} which is not integrable with the power $q\leqslant 3$ for large values of $|x|$.
Hence, Theorem~\ref{lemENSRegular} states that a far field asymptotic behavior of  stationary solutions corresponding to  integrable external forces satisfying \eqref{g1:force}
is described by the corresponding Sl\"ezkin-Landau solutions.
Similar results have been already obtained 
for  stationary solutions to the  Navier-Stokes system
on an exterior domain in \cite{KS11, KMT12} and 
on the whole space $\RRR$ in \cite{DI17}.
In Theorem \ref{lemENSRegular}, we generalize all those results by showing 
 that
 assumption $g_1-g_2\in \PM^{b-2}$ for some   $b\in (1,2)$
determines when two external forces from the space of pseudomeasure
$\PM^0$
({\it i.e.}~tempered distributions with bounded Fourier transforms) 
 lead to stationary solutions with the same asymptotic behavior as $|x|\to\infty$ .
\end{remark}

\begin{remark} 
On the other hand, Theorem \ref{lemENSRegular} for $b\in (2,3)$ 
allows us to study local singularities of stationary solutions. 
As an example, we consider again the  the Sl\"ezkin-Landau solution $U^\beta$
with a singularity at the origin which is not integrable with every  exponent $q>3$. By Theorem   \ref{lemENSRegular}, every solution $w_2$ with a small external force $g_2\in\PM^0$ such that 
\begin{equation}\label{g2beta}
\esssup_{\xi\in\RRR} |\xi|^{b-2}\big|\beta (2\pi)^{-3/2} -\widehat g_2(\xi)
\big|<\infty
\end{equation}
has a singularity at the origin which asymptotically resembles the singularity of $U^\beta$ and this  is measured by the $L^q$-norm with some $q>3$.
Note that, since $b\in (2,3)$,
 the assumption \eqref{g2beta} requires from the bounded function $\widehat g_2(\xi)$ to be close to the constant vector
 $\beta (2\pi)^{-3/2}=\beta \widehat{\delta_0}$ for large values of $|\xi|$.
\end{remark}

Finally, using the framework introduced above, we study an asymptotic stability of  solutions to the Cauchy problem \eqref{eqNSNS}
 corresponding to singular external forces.

\begin{theorem}\label{thm_self_similar}
Let $u_i$ with $i\in \{1,2\}$ be  solutions to the Cauchy problem \rf{eqNSNS}
with  initial data $u_{0i} \in \PM^2$ and external forces $f_i\in \X^0$
such that $\| u_{0i}\|_{\PM^2} <\ve$ and $\|f_i \|_{\X^0}<\ve$,
where $\ve>0$ is provided by  Proposition \ref{thmNSExistence}.
Suppose that there exists $\delta\in (0,1)$ such that 
\begin{equation}\label{f1f2delta}
\sup_{t>0} t^{\frac{\delta}{2}}\|f_1(t)-f_2(t)\|_{\PM^\delta}<\infty.
\end{equation}
Then, for every  $q\in \left(3,\frac{3}{1-\delta}\right)$ there exists a constant $C=C(q)>0$ such that 
\begin{equation}\label{u1u2:q}
\| u_1(\cdot, t) - u_2(\cdot,t)\|_q \leqslant Ct^{-\frac{1}{2}+\frac{3}{2q}} \qquad \textit{for all}\quad t>0.
\end{equation}
If, moreover, 
\begin{equation}\label{f1f2delta:1}
\lim_{t\to \infty} \|f_1(t)-f_2(t)\|_{\PM^0}=0
\quad \text{and}\quad \lim_{t\to \infty} \|S(t)\big(u_{01}-u_{02}\big)\|_{\PM^2}=0
\end{equation}
then
\begin{equation}\label{u1u2:dec}
\lim_{t\to\infty} t^{\frac{1}{2}-\frac{3}{2q}}
 \| u_1(\cdot, t) - u_2(\cdot,t)\|_q =0.
\end{equation}
\end{theorem}

\begin{remark}
In the particular case of a (sufficiently small) time independent  external force $f(\cdot, t)=g\in L^1(\RRR)^3$, Theorem \ref{thm_self_similar}
 states that each solution $u=u(x,t)$  of  the Cauchy problem \rf{eqNSNS}
with  an (sufficiently small) initial datum $u_{0} \in \PM^2$
converges as $t\to\infty$ towards a stationary solution satisfying system 
\eqref{eqENSNS} and this convergence holds true in the $L^q$-norm for each $q\in (3, \infty)$ (notice that condition \eqref{f1f2delta} is fulfiled for each $\delta \in (0,1)$ if $f_1=f_2$). 
Such an asymptotic stability result for Leray stationary solutions was obtained in \cite{BS09}.
We do not require from solutions to have a finite energy and our approach resembles the one  used in \cite{GHS97}.
\end{remark}

\begin{remark}
One should be careful with using the stability result from Theorem \ref{thm_self_similar} in the case of singular external forces.
Indeed, if $u_0\in\PM^2$ is homogeneous of degree $-2$ and $f\in\X^0$ is invariant under the scaling \eqref{f:scal}, then the corresponding solution is self-similar  \cite[Corollary 4.1]{CK04}:
$
u(x,t) =  t^{-1/2}u\left(x t^{-1/2}, 1\right)
$
for all $t>0$.
Thus, if moreover, $u(\cdot, 1)\in L^q(\RRR)$ then
$$
\|u(\cdot, t)\|_q= 
t^{-\frac{1}{2}+\frac{3}{2q}} \|u(\cdot, 1)\|_q 
\qquad \mbox{for all}\quad t>0.
$$
In such a case, inequality  \eqref{u1u2:q} does not say anything about the asymptotic behavior of solutions.
 However, in the particular case of $f=\beta_1 e_1 \delta_0$ the corresponding stationary solution $U^\beta$  is explicit \eqref{L sol}, homogeneous of degree $-1$ and self-similar. Moreover,  
 $U^\beta\notin L^q(\RRR)$ for each $q\in [1,\infty]$.
 In this case, inequality  \eqref{u1u2:q} states that each solution $u=u(x,t)$  of  the Cauchy problem \rf{eqNSNS}
with  an (sufficiently small) initial datum $u_{0} \in \PM^2$
and with the external force $f=\beta_1 e_1 \delta_0$
has a singularity at zero of the same shape as the Sl\"ezkin-Landau solution $U^{\beta_1}$. This approach was used in \cite{KZ15} in the study of singularities of solutions to problem \eqref{NSsing}.
Thus, additional assumptions \eqref{f1f2delta:1} are needed in the study of a large time behavior of singular solutions.
\end{remark}

\begin{remark}
Asymptotic results in Theorem \ref{thm_self_similar} cannot be true for 
any $q<3$ due to the estimates in Theorem \ref{corNSLandau}.
\end{remark}

\section{Review of estimates in $\PM^{a}$-spaces}\label{estimates}

First, we gather estimates which involve $\PM^a$-norms and which 
play a crucial role in the proofs of results from Section \ref{results}. Some of the following results can be found in the previous works \cite{CK04, KZ15, KSS20}. For a completeness of the exposition, we recall them in a form, which will be convenient in the proofs of results from this paper.

\subsection{Embedding theorems}

The following three lemmas state that tempered distributions from the space $\PM^a$ for a certain range of the exponent $a>0$, are in fact locally integrable functions.

\begin{lemma}\label{KPS-lem}
The following embedding holds true
\begin{equation*}
  \PM^{a}\subset L^{\frac{3}{3-a},\infty}(\RRR),
\qquad \text{for each}\quad a\in \left(\tfrac{3}{2},3\right),
\end{equation*}
where $L^{\frac{3}{3-a},\infty}(\RRR)$ is the  weak $L^p$-space (the Marcinkiewicz space). 
\end{lemma}

\begin{proof}
Let $w\in \PM^a$ for some $a\in \left(\frac{3}{2},3\right)$ and let $\vf \in C_c^\infty (\RRR)$ be an arbitrary function. Properties of the Fourier transform combined with the H\"older inequality in the Lorentz spaces (see e.g.~\cite{O'N63}) yield 
\begin{align*}
   \bigg| \int_{\RRR}w \vf \ud x \bigg|  &=\bigg| \int_{\RRR}\widehat {w} (\xi ) \overline{\widehat{\vf }(\xi )} \ud \xi \bigg| \\
   &\leqslant \| w\|_{{\PM}^a}\int_{\RRR} |\xi |^{-a}\big|\widehat{\vf }(\xi ) \big| \ud \xi \\
   &\leqslant \| w\|_{{\PM}^a}\| |\cdot |^{-a}\|_{L^{\frac{3}{a},\infty }}\| \widehat{\vf }\|_{L^{\frac{3}{3-a},1}}.
\end{align*}
Now, it suffices to apply the Hausdorff-Young inequality in the Lorentz spaces \cite{K97}
\begin{equation*}
   \|\widehat \vf\|_{L^{p',r}} \leqslant C\| \vf \|_{ L^{p,r}}, \quad \textrm{where}\ \ 1<p<2, \ 1\leqslant r<\infty , \quad \textrm{and }\ \ \tfrac{1}{p'}+\tfrac{1}{p}=1,
\end{equation*}
to obtain
\begin{equation*}
   \bigg| \int_{\RRR}w \vf  \ud x \bigg|  \leqslant C\| w\|_{{\PM}^a} \| \vf \|_{L^{\frac{3}{a}, 1}}
\end{equation*}
for all $\vf \in C^\infty_c(\RRR)$ and, by a density argument, for all $\vf \in L^{\frac{3}{a},1}(\RRR)$. Hence, the distribution $w$ defines a continuous linear functional on 
$ L^{\frac{3}{a},1}(\RRR)$, consequently, $w\in  L^{\frac{3}{3-a},\infty}(\RRR)$ (see e.g.~\cite[Thm. 1.4.17]{G14}).
\end{proof}

\begin{lemma}\label{lemEmbedding}
For each $b\in\left[0,2\right)$ and $q\in \left(\frac{3}{3-b},3\right)$ such that $q\geqslant 2$,
the following embedding holds true
$
  \PM^{2}\cap\PM^{b} \subset L^q(\RRR)  .
$
Moreover, there exists a constant $C=C(b,q)$ such that
\begin{equation}\label{PM:Lp}
  \|w\|_{L^{q}}\leqslant C\|w\|_{\PM^{2}}^{1-\frac{q-3}{q(b-2)}}\|w\|_{\PM^{b}}^{\frac{q-3}{q(b-2)}}.
\end{equation}
\end{lemma}

\begin{proof}
For each couple $(q,r)$ of positive numbers  satisfying 
 $$  
 \frac{1}{q}+\frac{1}{r}=1 
 \quad \text{with }\quad  q\geqslant 2\quad \text{and}\quad q\in \left(\frac{3}{3-b},3\right),
 $$ 
we obtain the inequalities    $2r>3$ and $br<3$. 
 Thus, by  the classical Hausdorff-Young inequality with these exponents, 
 we obtain
\begin{equation} \label{wLqLr}
  \begin{split}
  \|w\|_{L^{q}}^{r}\leqslant & C\|\widehat{w}\|_{L^{r}}^{r}\\
  = & C\int_{\RRR}|\widehat{w}(\xi)|^{r}\ud \xi\\
  = & C\int_{|\xi|\leqslant K}|\widehat{w}(\xi)|^{r}\ud \xi+C\int_{|\xi|\geqslant K}|\widehat{w}(\xi)|^{r}\ud \xi\\
  \leqslant & C\|w\|_{\PM^{b}}^{r}\int_{|\xi|\leqslant K}|\xi|^{-br}\ud \xi+C\|w\|_{\PM^{2}}^{r}\int_{|\xi|\geqslant K}|\xi|^{-2r}\ud \xi\\
  \leqslant & C\|w\|_{\PM^{b}}^{r}K^{3-br}+C\|w\|_{\PM^{2}}^{r}K^{3-2r}.
  \end{split}
\end{equation}
Now, we choose
\begin{equation*}
  K=\left(\frac{\|w\|_{\PM^{b}}}{\|w\|_{\PM^{2}}}\right)^{\frac{1}{b-2}}
  \quad\text{and}\quad r=\frac{q}{q-1}
\end{equation*}
to obtain inequality \eqref{PM:Lp}.
\end{proof}

\begin{remark}
Notice that if $b\in\left[0,3/2\right)$ and $q\in \left(\frac{3}{3-b},2\right)$, the proof of Lemma \ref{lemEmbedding} implies only that each 
$w\in  \PM^{2}\cap\PM^{b}$ satisfies 
$\widehat w\in L^\frac{q}{q-1}(\RRR)$.
\end{remark}

\begin{lemma}\label{lemEmbedding2}
For each $\delta \in(0,1)$ and $q\in \left(3,\frac{3}{1-\delta} \right)$,
the following embedding holds true
$
  \PM^{2}\cap\PM^{2+\delta } \subset L^q(\RRR)  .
$
Moreover, there exists a constant $C=C(\delta,q)$ such that
\begin{equation}\label{PM:Lp:2}
  \|w\|_{L^{q}}\leqslant C\|w\|_{\PM^{2}}^{1-\frac{q-3}{\delta q}}\|w\|_{\PM^{2+\delta}}^{\frac{q-3}{\delta q}}.
\end{equation}
\end{lemma}
\begin{proof}
Here, one should proceed analogously as in the proof of Lemma \ref{lemEmbedding} and the only difference consists in a 
suitable modification of exponents 
in inequality \eqref{wLqLr}; see also \cite[Lemma 7.4]{CK04} and \cite[Lemma 3.6]{KZ15} for analogous  calculations. 
\end{proof}

\begin{proposition}\label{est:moment}
Let $b\in (1,2)$.
For every $g\in L^1(\RRR)$ such that $\int_{\RRR} |x|^{2-b} |g(x)|\ud x <\infty$ and for $\beta =\int_{\RRR} g(x)\ud x$ we have the estimate
$$
\sup_{\xi\in\RRR} |\xi|^{b-2} \big|\widehat g(\xi)-\beta\big|\leqslant
C(b) \int_{\RRR} |x|^{2-b} |g(x)|\ud x 
$$
with a constant $C(b)>0$ independent of $g$. 
\end{proposition}

\begin{proof}
The estimate is immediate in the case $b=1$. Indeed, since $\beta = \widehat g(0)$, by the mean value theorem and properties of the Fourier transform, we obtain 	
\begin{equation*}
	|\xi |^{-1}|\widehat g(\xi) - \widehat g(0)| \leqslant \| \nabla\widehat g \|_{\infty} \leqslant (2\pi)^{-\frac{3}{2}} \int_{\RRR} |x| |g(x)|\ud x.
\end{equation*}	
In a more general case, we use the well-known formula (see {\it e.g.}~\cite[Ch.V, Lemma 2]{Stein}) 
$$
\frac{1}{|\xi|^{2-b}} = C(b) \int_{\RRR} e^{-ix\cdot \xi} \frac{1}{|x|^{1+b}}\ud x
$$
valid for every $b\in (-1,2)$ and a constant $C(b)>0$.
Thus 
\begin{equation*}
\begin{split}
|\xi|^{b-2} \big|\widehat g(\xi)-\beta\big|&=
C(b)\left|
\int_{\RRR} e^{-ix\cdot \xi} \int_{\RRR} \left(
\frac{1}{|x-y|^{1+b}}-\frac{1}{|x|^{1+b}}
\right) g(y)\ud y \ud x
\right|\\
&\leqslant C(b)
\int_{\RRR} \int_{\RRR} \left|
\frac{1}{|x-y|^{1+b}}-\frac{1}{|x|^{1+b}}
\right| \ud x |g(y)|\ud y.
\end{split}
\end{equation*}
Note that the integral with respect to $x$ is finite for every $y\in\RRR$ because its integrand $\big| |x-y|^{-1-b}-|x|^{-1-b}\big|$ is locally integrable and behaves like $|x|^{-2-b}$ when $|x|\to \infty$ (here, the assumption $b\in (1,2)$ is crucial). Hence, by the change of variables $x=\omega |y|$, it follows that 
$$
\int_{\RRR} \left|
\frac{1}{|x-y|^{1+b}}-\frac{1}{|x|^{1+b}}
\right| \ud x
=
|y|^{2-b} 
\int_{\RRR} \left|
\frac{1}{|\omega-y/|y||^{1+b}}-\frac{1}{|\omega|^{1+b}}
\right| \ud \omega.
$$
Finally,  for $b\in (1,2)$,  we have got  
$
\sup_{y\in\RRR}\int_{\RRR} \left|
{|\omega-y/|y||^{-1-b}}-{|\omega|^{-1-b}}
\right| \ud \omega<\infty 
$
which completes the proof of the proposition.
\end{proof}

\subsection{Estimates of the heat semigroup}

First, we recall properties of the heat semigroup \rf{heat_semigroup} which satisfies the well-known formula
\begin{equation}\label{Fourier_heat}
\widehat{S(t)u_0}(\xi)=e^{-t|\xi|^2}\widehat u_0(\xi) .
\end{equation}

\begin{lemma}\label{lemm_heat_PM}
For each $\delta \geqslant 0$ and for every $u_0 \in \PM^2$ there exists a constant $C=C(\delta)>0$ such that the following estimate holds true
\begin{equation}\label{heat_PMa_est}
\| S(t)u_0\|_{\PM^{2+\delta}} \leqslant Ct^{-\frac{\delta}{2}}\| u_0\|_{\PM^2} \quad \textit{for all}\quad t>0.
\end{equation}
\end{lemma}

\begin{proof}
It follows from equation \rf{Fourier_heat} that
\begin{align*}
\| S(t)u_0\|_{\PM^{2+\delta}} &= \esssup_{\xi \in \RRR}|\xi|^{2+\delta} e^{-t|\xi|^2}|\widehat{u_0}(\xi)| \leqslant  \|u_0\|_{\PM^2}t^{-\frac{\delta}{2}}\esssup_{\xi \in \RRR}|\sqrt{t}\xi|^{\delta}e^{-|\sqrt{t}\xi|^2} \\
&= C(\delta)t^{-\frac{\delta}{2}}\|u_0\|_{\PM^2}
\end{align*}
since the function $|\sqrt{t}\xi|^{\delta}e^{-|\sqrt{t}\xi|^2}$ is bounded 
 for each $\delta\geqslant 0$
with respect to  $\xi\in \RRR$ and $t\geqslant 0$.
\end{proof}
\subsection{Estimates of bilinear forms}\label{Est_of_bili_form}

Next, we gather estimates of the bilinear forms $B(\cdot,\cdot)$ from \rf{bili_form} and $B_E(\cdot, \cdot)$ defined in \rf{bili_form_E}, where both quantities should be defined via the Fourier transform as follows. For time dependent tempered distributions $u(\cdot, t)$ and $v(\cdot,t)$ we put
\begin{equation}\label{F_tran_B}
  \widehat{B(u,v)}(\xi,t)=\int_{0}^{t}e^{-(t-\tau)|\xi|^{2}}i\xi \widehat\Projector (\xi)\int_{\RRR}\widehat{u}(\eta,\tau)\otimes \widehat{v}(\xi-\eta,\tau)\ud  \eta\ud \tau.
\end{equation}
and for time independent $w$ and $z$, we set
\begin{equation}\label{F_tran_BE}
\begin{split}
  \widehat{B_{E}(w,z)}(\xi)&=\int_{-\infty}^{t}e^{-(t-\tau)|\xi|^{2}}i\xi \widehat{\Projector}(\xi)\int_{\RRR}\widehat{w}(\eta)\otimes\widehat{z}(\xi-\eta)\ud  \eta\ud \tau \\
&= \frac{i\xi}{|\xi|^2} \widehat{\Projector}(\xi)\int_{\RRR}\widehat{w}(\eta)\otimes\widehat{z}(\xi-\eta)\ud  \eta,
\end{split}
\end{equation}
where we use the simple formula 
$$
\int_{-\infty}^t e^{-(t-s)|\xi|^2}\ud s =\frac{1}{|\xi|^2}\qquad \text{for each} \quad 
\xi \in\RRR\setminus \{0\}
$$ 
to obtain the second equality in \rf{F_tran_BE}. Below, we estimate both bilinear forms in the norms of  $\PM^a$-spaces using the equation
(see {\it e.g} \cite[Ch.~V, Sec.~1]{Stein})
\begin{equation}\label{int:C(b)}
\int_{\RRR} \frac{1}{|\eta|^2}\frac{1}{|\xi -\eta|^b} \ud \eta = \frac{C(b)}{|\xi |^{b-1}}\qquad \text{for all}\quad \xi \in \RRR\setminus \{0\}
\end{equation}
which holds true for every $b\in (1,3)$ and a constant $C(b)>0$ independent of $\xi$.

\begin{lemma}\label{est_of_bil_forms}
There exist constants $\eta >0$ and $\eta_E >0$ such that for all $u, v \in \mathcal{X}^2$ and $w,z \in \PM^2$
\begin{equation}\label{est_B_form}
\| B(u,v)\|_{\mathcal{X}^2}\leqslant \eta \| u\|_{\mathcal{X}^2}\| v\|_{\mathcal{X}^2}
\end{equation}
and
\begin{equation}\label{est_BE_form}
\| B_E(w,z)\|_{\PM^2}\leqslant \eta_E \| w\|_{\PM^2}\| z\|_{\PM^2}
\end{equation}
\end{lemma}
\begin{proof}
Both inequalities are an immediate consequence of equation \eqref{int:C(b)} with $b=2$ and a slightly more general   reasoning is used in the next  proof below.
\end{proof}

\begin{lemma}\label{lemma_est_B_PMb}
For each $b \in \left[0,2\right]$ there exists a constant $C=C(b)>0$ such that the following inequality holds true for all $u\in \mathcal{X}^2$ 
\begin{equation}\label{est_B_PMb}
\| B(u,u)(t)\|_{\PM^b}\leqslant Ct^{\frac{2-b}{2}}\| u\|_{\mathcal{X}^2}^2 \qquad \textit{for all}\quad t>0 .
\end{equation}
\end{lemma}

\begin{proof}
By equation \eqref{int:C(b)}, we have got the following estimate
\begin{align}\label{ksi_u_est}
  \left|\int_{\RRR}\widehat{u}(\eta,\tau)\otimes \widehat{u}(\xi-\eta,\tau)\ud  \eta\right| \leqslant \|u\|_{\XPM^{2}}^{2}\int_{\RRR}\frac{1}{|\eta|^{2}}\frac{1}{|\xi-\eta|^{2}}\ud  \eta  = \frac{C}{|\xi|}\|u\|_{\XPM^{2}}^{2}
\end{align}
which by formula \rf{F_tran_B} implies 
\begin{align*}
  \| B(u,u)(t)\|_{\PM^{b}}&\leqslant  C\|u\|_{\XPM^{2}}^{2}\sup_{\xi\in\RRR}|\xi|^{b}\int_{0}^{t}e^{-(t-\tau)|\xi|^{2}}\ud \tau\\
  &= C\|u\|_{\XPM^{2}}^{2}\sup_{\xi\in\RRR}\frac{1-e^{-t|\xi|^{2}}}{|\xi|^{2-b}}
  = Ct^{\frac{2-b}{2}}\|u\|_{\XPM^{2}}^{2}
\end{align*}
for all $t>0$.
\end{proof}

\begin{lemma}\label{lemma_est_BE_form}
For each $b \in \left(1,3\right)$ there exists a constant $C=C(b)>0$ such that for all $w\in \PM^2$ and $z \in \PM^b$ we have got the inequality 
\begin{equation}\label{est_BE_Yb}
\| B_E(w,z)\|_{\PM^b}\leqslant C(b) \| w\|_{\PM^2}\| z\|_{\PM^b}.
\end{equation}
\end{lemma}

\begin{proof}
By equation \rf{int:C(b)}, we obtain the inequality
\begin{align*}
  \left|\int_{\RRR}\widehat{w}(\eta,\tau)\otimes\widehat{z}(\xi-\eta,\tau)\ud  \eta\right| &\leqslant \| w\|_{\PM^{2}}\|z\|_{\PM^{b}}\int_{\RRR}\frac{1}{|\eta|^{2}}\frac{1}{|\xi-\eta|^{b}}\ud  \eta\\
&=\frac{C}{|\xi|^{b-1}}\| w\|_{\PM^{2}}\| z\|_{\PM^{b}}
\end{align*}
which, by the formula on the right hand side of 
equation \rf{F_tran_BE}, leads to inequality~\rf{est_BE_Yb}.~\end{proof}

\begin{lemma}\label{lemma_2delta_norm}
For each $\delta \in [0,1)$ there exists a constant $C=C(\delta )>0$ such that for all $u\in \mathcal{X}^2$ and $v(t) \in \PM^{2+\delta}$ such that $\sup_{t>0} t^{\frac{\delta}{2}}\| v(t)\|_{\PM^{2+\delta}}<\infty$, we have got 
\begin{equation*}
\sup_{t>0}t^{\frac{\delta}{2}}\| B(u,v)\|_{\PM^{2+\delta}}\leqslant C \| u\|_{\mathcal{X}^2}\sup_{t>0}t^{\frac{\delta}{2}}\| v(t)\|_{\PM^{2+\delta}}.
\end{equation*}
\end{lemma}

\begin{proof}
By equation \eqref{int:C(b)} with $b=2+\delta$, we have got the following estimate
\begin{align*}
  \left|\int_{\RRR}\widehat{u}(\eta,\tau)\otimes \widehat{v}(\xi-\eta,\tau)\ud  \eta\right| &\leqslant \|u(\tau )\|_{\PM^{2}}\|v(\tau )\|_{\PM^{2+\delta}}\int_{\RRR}\frac{1}{|\eta|^{2}}\frac{1}{|\xi-\eta|^{2+\delta}}\ud  \eta  \\
&\leqslant \frac{C}{|\xi|^{1+\delta}}\|u(\tau)\|_{\PM^2}\|v(\tau)\|_{\PM^{2+\delta}}
\end{align*}
which, by formula \rf{F_tran_B}, implies 
\begin{align*}
  t^{\frac{\delta}{2}}\| B(u,v)\|_{\PM^{2+\delta}}&\leqslant  Ct^{\frac{\delta}{2}}\|u\|_{\XPM^{2}}\sup_{t>0}t^{\frac{\delta}{2}}\|v(t)\|_{\PM^{2+\delta}}\sup_{\xi\in\RRR}\int_{0}^{t}|\xi|^2e^{-(t-\tau)|\xi|^{2}}\tau^{-\frac{\delta}{2}}\ud \tau
\end{align*}
for all $t>0$. To complete the proof, we show that the quantity 
\begin{equation*}
t^{\frac{\delta}{2}}\int_{0}^{t}|\xi|^2e^{-(t-\tau)|\xi|^{2}}\tau^{-\frac{\delta}{2}}\ud \tau
\end{equation*}
is bounded by a constant independent of $\xi$ and $t$. Indeed, we split the integral  with respect to $\tau$ into two parts and we begin with 
\begin{align*}
I_1 &= t^{\frac{\delta}{2}}|\xi|^2 \int_0^{\frac{t}{2}}e^{-(t-\tau)|\xi|^{2}}\tau^{-\frac{\delta}{2}}\ud \tau \leqslant t^{\frac{\delta}{2}}|\xi|^2e^{-\frac{t|\xi|^2}{2}}\int_0^{\frac{t}{2}}\tau^{-\frac{\delta}{2}}\ud \tau
= C(\delta)\left|\sqrt{t}\xi\right|^2e^{-\frac{|\sqrt{t}\xi|^2}{2}},
\end{align*}
where  the function on the right-hand side is bounded for $\xi \in \RRR$ and $t>0$. 

Next, we deal with 
\begin{align*}
I_2 &= t^{\frac{\delta}{2}} \int_{\frac{t}{2}}^t|\xi|^2 e^{-(t-\tau)|\xi|^{2}}\tau^{-\frac{\delta}{2}}\ud \tau \leqslant C(\delta) \int_{\frac{t}{2}}^t  |\xi|^2 e^{-(t-\tau)|\xi|^2}\ud \tau = C(\delta) \left(1-e^{-\frac{t|\xi|^2}{2}}\right) \leqslant C(\delta)
\end{align*}
which completes the proof.
\end{proof}

\subsection{Estimates of external forces}

Next, we are going to estimated the quantities which depend on external forces: the term $F=F(x,t)$ defined in \eqref{force} and 
the term $G=G(x,t)$ given by formula \eqref{F_E}. Here, again, we define these formulas by the Fourier transform as follows
\begin{equation}\label{F_trans_F}
\widehat{F}(\xi, t) = \int_0^t e^{-(t-\tau)|\xi|^2}\widehat{\Projector}(\xi)\widehat{f}(\xi, \tau) \ud \tau
\end{equation}
and
\begin{equation}\label{F_trans_G}
\widehat{G}(\xi ) = \int_{-\infty}^t e^{-(t-\tau)|\xi|^2}\widehat{\Projector}(\xi)\widehat{g}(\xi ) \ud \tau =\frac{1}{|\xi|^2}\widehat{\Projector }(\xi)\widehat{g}(\xi).
\end{equation}

\begin{lemma}\label{lemma_est_F}
For each $b\in \left[0,2\right]$ there exists a constant $C=C(b)>0$ such that for each $f\in \mathcal{X}^0$ the following inequality holds true
\begin{equation}\label{est_F}
\| F(t)\|_{\PM^b}\leqslant C(b)t^{\frac{2-b}{2}}\|f\|_{\mathcal{X}^0}\qquad \textit{for all}\quad t>0.
\end{equation}
\end{lemma}

\begin{proof}
The reasoning is analogous to that  in the proof of 
Lemma \ref{lemma_est_B_PMb}. Using the definition of $F(t)$ in \rf{F_trans_F} and of the norm in $\PM^b$ in \rf{PM_space} we obtain
\begin{align*}
  \|F(t)\|_{\PM^b}&\leqslant \esssup_{\xi \in \RRR} |\xi|^b\int_0^t e^{-(t-\tau)|\xi|^2}|\widehat{\Projector}(\xi)| |\widehat{f}(\xi,\tau)| \ud \tau \\
&\leqslant 2 \|f\|_{\mathcal{X}^0}\sup_{\xi \in \RRR}\frac{1-e^{-t|\xi|^2}}{|\xi|^{2-b}}= C(b)t^{\frac{2-b}{2}}\|f\|_{\mathcal{X}^0}.
\end{align*}
\end{proof}

\begin{corollary}\label{cor:F:Lq}
For each $q\in \left[2,3\right)$ there exists a constant $C>0$ such that for each $f\in \mathcal{X}^0$ the following inequality holds true
\begin{equation}\label{F_lq}
	\| F(t)\|_{L^q}\leqslant Ct^{\frac{3-q}{2q}}\|f\|_{\mathcal{X}^0}\qquad \textit{for all}\quad t>0.
\end{equation}
\end{corollary}

\begin{proof}
We combine estimate \rf{est_F} for $b=2$ and for $b=0$ with 
 the interpolation inequality from Lemma \ref{lemEmbedding} 
in order to obtain the following estimate for each $q\in [2,3)$
\begin{align*}
		\| F(t)\|_{L^q}\leqslant C\|F(t)\|_{\PM^2}^{1-\frac{3-q}{2q}}
		\|F(t)\|_{\PM^{0}}^{\frac{3-q}{2q}}\leqslant C  t^{\frac{3-q}{2q}}
		\|f\|_{\mathcal{X}^0} 
\end{align*}
for all $t>0$.
\end{proof}

\begin{lemma}\label{lemma_f_2delta}
For each $\delta\in  [0,1)$ there exists a constant $C(\delta)>0$ such that for every  $f=f(\cdot,t) $ 
satisfying 
$\sup_{t>0} t^{\frac{\delta}{2}}\|f(t)\|_{\PM^\delta}<\infty$
 we obtain
\begin{equation*}
 \sup_{t>0} t^{\frac{\delta}{2}} \left\| \int_{0}^t S(t-\tau ) \Projector f(\tau) \ud \tau \right\|_{\PM^{2+\delta}}\leqslant C(\delta)
 \sup_{t>0} t^{\frac{\delta}{2}}
 \|f(t)\|_{\PM^\delta} .
\end{equation*}
\end{lemma}

\begin{proof}
By the definition of the $\PM^{\delta+2}$-norm, we have got the inequality
\begin{align*}
t^{\frac{\delta}{2}} \left\| \int_{0}^t  S(t-\tau )\Projector f(\tau) \ud \tau \right\|_{\PM^{2+\delta}} \leqslant &2\sup_{t>0}t^{\frac{\delta}{2}}\| f(\tau)\|_{\PM^\delta} \\
&\times 
\sup_{\xi \in \RRR,t>0}
t^{\frac{\delta}{2}}|\xi |^{2} \int_{0}^t e^{-(t-\tau )|\xi |^2}\tau^{-\frac{\delta}{2}}\ud \tau, 
\end{align*}
where the second factor on the right-hand side is finite by the same argument as the one used in the proof of Lemma \ref{lemma_2delta_norm}.
\end{proof}

\begin{lemma}\label{Lem:Tf}
For every $\vf \in \mathcal{S}(\RRR)$ we define
\begin{equation}\label{t_d}
	T\vf = \int_{-\infty}^{\infty} \log |x_3| \partial_{x_3}\vf (0,0,x_3)\ud x_3.
\end{equation}	
Then $T\in \mathcal{S}'(\RRR)$ and 
\begin{equation}\label{lttd}
	\widehat T(\xi) = 2^{-\frac{3}{2}} \pi^{-\frac{1}{2}}i \sgn \xi_3.
\end{equation}
\end{lemma}

\begin{proof}
We repeat the well-known calculations from the study of the principal value distribution (see {\it e.g.}~\cite[Ch.3]{L58}). First, we show that $T$ defined by formula \rf{t_d} is a tempered distribution. Indeed, for $\vf \in \mathcal{S}(\RRR)$ integrating by parts, we have 
\begin{align*}
|T\vf |&= \left|\lim_{\ve \to 0}\int_{|x_3|>\ve } \log |x_3| \partial_{x_3}\vf (0,0,x_3)\ud x_3 \right| \\
& =\left|\lim_{\ve \to 0}\int_{|x_3|>\ve } \frac{1}{x_3} \vf (0,0,x_3)\ud x_3
+\lim_{\ve \to 0}\big( \log |\ve| \left(\vf (0,0,-\ve)-\vf(0,0,\ve)\right)\big) \right|\\
&=\left|\lim_{\ve \to 0}\int_{\ve <|x_3|<1 } \frac{1}{x_3}\vf (0,0,x_3)\ud x_3 +\int_{|x_3|\geqslant 1 } \frac{1}{x_3}\vf (0,0,x_3)\ud x_3 \right|\\
&\leqslant \| \vf'\|_{\infty} +\int_{|x_3|\geqslant 1 } \frac{1}{|x_3|}|\vf (0,0,x_3)|\ud x_3 <\infty,
\end{align*}	
because the boundary term, which result from integration by parts, tends to zero as $\ve$ goes to zero. Since $i\pi \sgn \xi_3$ is a Fourier transform of the one-dimensional principal value distribution $-\frac{1}{x_3}$ (see {\it e.g.}~\cite[Ch.3]{L58}) applying the three-dimensional definition of the Fourier transform, we obtain
\begin{align*}
	T\widehat \vf &= -\lim_{\ve \to 0}\int_{|x_3|>\ve } \frac{1}{x_3} \widehat \vf (0,0,x_3)\ud x_3 =- (2\pi)^{-\frac{3}{2}}\lim_{\ve \to 0} \int_{\RRR} \int_{|x_3|>\ve } \frac{1}{x_3} e^{-i x\cdot \xi} \vf(\xi_1,\xi_2,\xi_3) \ud x_3 \ud \xi\\
	&=(2\pi)^{-\frac{3}{2}}\int_{\RRR}i\pi \sgn \xi_3 \vf (\xi_1,\xi_2,\xi_3)\ud \xi.
\end{align*}
This proves formula \rf{lttd}.
\end{proof}	
\section{Proofs of results from Section \ref{results}}\label{proofs}

Solutions to integral equations \eqref{int_eq} and  \eqref{int_E} are obtained from the Banach fixed point theorem which, in the case of ``quadratic'' equations, is often reformulated as in the following lemma. We skip its elementary proof.

\begin{lemma}\label{lem:xyB}
Let $(X,\|\cdot\|)$ be a Banach space, $L:\,X\to X$ be a linear bounded operator such that for a constant $\lambda\in[0,1)$, we have
$\|L(x)\|\leqslant \lambda\|x\|$ for all $x\in X,$
and $B :X \times X \to X$ be a bilinear mapping such that
$$\|B(x_1,x_2)\|\leqslant \eta\|x_1\|\,\|x_2\|\quad\text{for every}\quad x_1,\, x_2\in X $$
for some constant $\eta>0$. Then, for every $y\in X$ satisfying $4\eta\|y\|< (1-\lambda)^{2}$, the equation
\begin{equation}\label{eq.fixed-theorem}
x = y +L(x)+B(x,x)
\end{equation}
 has a solution $x\in X$. In particular, this solution satisfies $\|x\|\leqslant \frac{2\|y\|}{1-\lambda}$,
and it is the only one among all solutions satisfying $\|x\|< \frac{1-\lambda}{2\eta}$.
\end{lemma}

 Results from Section \ref{results} are obtained immediately from estimates in Section \ref{estimates}.

\begin{proof}[Proof of Proposition \ref{thmNSExistence}] 
It follows from Lemma \ref{lemm_heat_PM} that 
$\|S(\cdot)u_0\|_{\X^2}\leqslant \|u_0\|_{\PM^2}$ for each $ 
u_0\in \PM^2$. Moreover, by Lemma \ref{lemma_est_F} with $b=2$, we obtain 
$\|F\|_{\X^2}\leqslant 2\|f\|_{\X^0}.$
Thus, 
in order to construct a solution to equation \eqref{int_eq},
it suffices to apply Lemma \ref{lem:xyB} together with estimate \rf{est_B_form}.
\end{proof}

\begin{proof}[Proof of Theorem \ref{corNSLandau}] 
First we show that
every solution $u\in \mathcal{X}^2$ of problem \rf{eqNSNS} corresponding to an initial datum 
$u_0\in \PM^{2}$ and an external force $f \in \mathcal{X}^0$ satisfies 
$$
u(t)-S(t)u_0-F(t) \in \PM^b\qquad 
 \text{for each}\quad  t>0\quad  \text{and}\quad  b\in \left[0,2\right].
 $$ 
Indeed, 
applying Lemma~ \ref{lemma_est_B_PMb} 
to  integral equation \rf{int_eq}
we have got
\begin{align}\label{uSbBF}
\|u(t) - S(t)u_0-F(t)\|_{\PM^b} &\leqslant \| B(u,u)(t)\|_{\PM^b}  \leqslant 
Ct^{\frac{2-b}{2}}\|u\|_{\mathcal{X}^2}^2
\end{align}
 for each $b\in \left[0,2\right]$ and for all $t>0$.
Next, 
 we combine the interpolation inequality from Lemma \ref{lemEmbedding} with the exponents $b=0$ and 
$q\in [2,3)$ together with  estimate \eqref{uSbBF} for $b=2$ and for $b=0$
 in order to obtain 
 \begin{align*}
 \|u(t)-S(t)u_0-F(t)\|_q &
 \leqslant C\|u(t)-S(t)u_0-F(t)\|_{\PM^2}^{1-\frac{3-q}{2q}}
 \|u(t)-S(t)u_0-F(t)\|_{\PM^{0}}^{\frac{3-q}{2q}}\\
 &\leqslant C  t^{\frac{3-q}{2q}}
 \|u\|^2_{\mathcal{X}^2} 
 \end{align*}
 for all $t>0$.
\end{proof}

\begin{proof}[Proof of Proposition \ref{thmENSExistence}]
Here, it suffices to proceed analogously as in the proof of 
Proposition~\ref{thmNSExistence} using the estimate of the bilinear form $B_E$ from \rf{est_BE_form}.
\end{proof}

\begin{proof}[Proof of Theorem \ref{lemENSRegular}]
We recall  that, by Proposition \ref{thmENSExistence}, both solutions to integral equation \rf{int_E}: the vector fields  $w_i=w_i(x)$ corresponding to the external forces $g_i\in\PM^0$ with $i\in \{1,2\}$ satisfy the inequalities
\begin{equation}\label{wUe}
\| w_i \|_{\PM^{2}}<C\ve \quad \text{with}\quad i\in \{1,2\} .
\end{equation}
First, we are going to use these inequalities  in order to 
 show that 
$w_1-w_2 \in \PM^{b}$  for every $b\in \left(1,3\right)$ together with the estimate 
\begin{equation}\label{w-U}
\| w_1-w_2 \|_{\PM^{b} }\leqslant C
\esssup_{\xi\in\RRR} |\xi|^{b-2}\big|\widehat g_1(\xi)-\widehat g_2(\xi)\big|=C \| g_1-g_2\|_{\PM^{b-2}}
.
\end{equation}
Indeed, 
using integral equations \rf{int_E} 
for both stationary solutions 
$w_1$ and $w_2$ we obtain that the difference  $W= w_1-w_2$ satisfies the equation 
\begin{equation*}
  W+B_{E}(w_1,W)(t)+B_{E}(W,w_2)(t)=\int_{-\infty}^t S(t-\tau )\Projector \big(g_1-g_2\big) \ud  \tau.
\end{equation*}
Applying Lemma   \ref{lemma_est_BE_form} 
and expression \eqref{F_trans_G}
we get the estimate
\begin{equation*}
  \|W\|_{\PM^{b}}\leqslant C\left(\| w_1 \|_{\PM^{2}}+\|w_2\|_{\PM^{2}}\right)\|W\|_{\PM^{b}}+2\esssup_{\xi\in\RRR} |\xi|^{b-2}\big|\widehat g_1(\xi)-\widehat g_2(\xi)
\big|
\end{equation*}
 which together with  inequalities \eqref{wUe} with sufficiently small $\ve>0$,
 completes the proof of inequality \eqref{w-U}. 

Finally, 
in order to estimate the $L^q$-norms of $w_1-w_2$,
it suffices to combine the estimates of the $\PM$-norms from   \eqref{w-U}   with the interpolation 
inequalities
from Lemma \ref{lemEmbedding} if $b\in (1,2)$ and from Lemma \ref{lemEmbedding2} with  $\delta = b-2$ if $b\in (2,3)$. 
\end{proof}

\begin{proof}[Proof of Theorem \ref{thm_self_similar}] 
This theorem  is, in fact, a minor extension of \cite[Thm.~7.2]{CK04}.
Using integral formula \rf{int_eq}  we obtain that the difference $V(t)=u_1(t)-u_2(t)$ satisfies the following equation
\begin{equation*}
\begin{split}
V(t) = &S(t)\big(u_{01}-u_{02}\big)-B(V,u_1)(t) - B(u_2 , V)(t) + \int_{0}^t  S(t-\tau) \Projector \big(f_1(\tau)-f_2(\tau)\big) \ud \tau. 
\end{split}
\end{equation*}
We compute the $\PM^{2+\delta}$-norm with $\delta\in [0,1)$  of this equation applying the estimates from Lemmas \ref{lemm_heat_PM},  
\ref{lemma_2delta_norm} and \ref{lemma_f_2delta} in order to obtain   
\begin{align*}
 t^{\frac{\delta}{2}}\| V(t)\|_{\PM^{2+\delta}}\leqslant & C\| u_{01}-u_{02}\|_{\PM^2}+C\big( \|u_1\|_{\mathcal{X}^2}+\| u_2\|_{\X^2}\big)\sup_{t>0}t^{\frac{\delta}{2}}\| V(t)\|_{\PM^{2+\delta}}\\
&+C\sup_{t>0} t^{\frac{\delta}{2}}\|f_1(t)-f_2(t)\|_{\PM^\delta} .  
\end{align*}
Multiplying the both sides of this estimate  by $t^{-\frac{\delta}{2}}$ and using 
the inequalities 
\begin{equation}\label{u234}
\|u_i\|_{\mathcal{X}^2}\leqslant C\left(\|u_{0i}\|_{\PM^2}+\|f_i\|_{\X^0}\right)\leqslant 2C\ve
\end{equation} 
from Proposition  \ref{thmNSExistence}
 with sufficiently small $\ve >0$, we complete the proof of the estimate 
\begin{equation}\label{u1u2:proof}
 \| u_1(\cdot , t) -u_2(\cdot,t) \|_{\PM^{2+\delta}}\leqslant C t^{-\frac{\delta}{2}} \qquad \text{for all}\quad t>0
\end{equation}
with a number $C>0$ depending on $\| u_{0i}\|_{\PM^2}$ and $\| f_i\|_{\X^0}$ but independent of $t$.
Now, in order to complete the proof, it suffices to combine this inequality (which holds true for every $\delta \in [0,1)$) with the  interpolation inequality \eqref{PM:Lp:2}.

In order to show the faster convergence in \eqref{u1u2:dec} under assumptions
\eqref{f1f2delta:1}, it suffices to  use the  interpolation inequality \eqref{PM:Lp:2}
together with estimate \eqref{u1u2:proof}
and with the relation $\lim_{t\to\infty} 
 \| u_1(\cdot , t) -u_2(\cdot,t) \|_{\PM^{2}}=0$ proved in \cite[Thm.~5.1]{CK04}.
\end{proof}




\end{document}